\documentclass[12pt, oneside]{amsart}
\usepackage{fullpage}
\usepackage{cite}
\usepackage[utf8]{inputenc}
\usepackage[T1]{fontenc}
\usepackage[english]{babel}
\usepackage{amsmath,amsfonts,amsthm, mathrsfs,amssymb}
\usepackage[linktocpage=true]{hyperref}
\usepackage{tikz-cd}
\usepackage{array}
\usepackage{ulem}

\theoremstyle{plain}
\newtheorem{theorem}{Theorem}[section]

\newtheorem{lemma}[theorem]{Lemma}
\newtheorem{proposition}[theorem]{Proposition}

\theoremstyle{definition}

\newtheorem{remark}[theorem]{Remark}
\numberwithin{equation}{section}
\numberwithin{figure}{section}
\numberwithin{table}{section}

\DeclareMathOperator{\rk}{rk}
\DeclareMathOperator{\Par}{Par}

\newcommand{\Z}[0]{\mathbb Z}

\newcommand{\x}{\mathbf{x}}
\newcommand{\y}{\mathbf{y}}
\newcommand{\la}{\lambda}

\usetikzlibrary{shapes.multipart}
\usetikzlibrary{patterns}
\usetikzlibrary{shapes.multipart}
\usetikzlibrary{arrows}
\usetikzlibrary{decorations.markings}
\usepgflibrary{decorations.shapes}
\usetikzlibrary{decorations.shapes}
\usepgflibrary{shapes.symbols}
\usetikzlibrary{shapes.symbols}
\usetikzlibrary{decorations.pathreplacing}

\tikzstyle{fleche}=[>=stealth', postaction={decorate}, thick]
\tikzstyle{axis}=[->, >=stealth', thick, gray]
\tikzstyle{path}=[->, >=stealth', thick]
\tikzstyle{grille}=[dotted, gray]
\colorlet{gray}{white!85!black}
\colorlet{red}{white!30!red}
\colorlet{blue}{white!75!blue}
\hypersetup{colorlinks=true, linkcolor=black}
\begin{document}
\title{\Large Measures and generalizations of dual Littlewood identities}
\author{Zhongren Cai}
\address{School of Science, Huzhou Normal University, Huzhou, Zhejiang 313000, China}
\email{2253583891@qq.com}
\author{Bin Jiang}
\address{School of Science, Huzhou Normal University, Huzhou, Zhejiang 313000, China}
\email{1520162385@qq.com}
\author{Naihuan Jing}
\address{Department of Mathematics, North Carolina State University, Raleigh, NC 27695, USA}
\email{jing@ncsu.edu}
\author{Zhijun Li$^\dag$}
\address{School of Science, Huzhou Normal University, Huzhou, Zhejiang 313000, China}
\email{zhijun1010@163.com}
\author{Qianyi Ye}
\address{School of Science, Huzhou Normal University, Huzhou, Zhejiang 313000, China}
\email{99575004@qq.com}
\thanks{{\scriptsize
\hskip -0.6 true cm MSC (2020): Primary: 05E05; Secondary: 17B69, 17B37, 82B20, 05A17.
\newline $^\dag$ Corresponding author: zhijun1010@163.com
}}
\maketitle
\begin{abstract} We introduce three families of vectors $|\underline{\lambda}^{so}\rangle$, $|\underline{\lambda}^{sp}\rangle$
and $|\underline{\lambda}^{o}\rangle$ parametrized by partitions in the Fock space by using
products of adjoint vertex operators. We show that the quotient space of
the dual vacuum vector is spanned by the partition vectors indexed by a special family of
partitions. The partition-indexed
vectors also help us to derive the dual Littlewood identities of types B, C, and D in a new manner associated to
the special family of partitions.  As an application, we obtain a new free fermionic 
construction to show that the measures related to dual Littlewood identities introduced by Rains \cite[Section 7]{Rai2000} and Betea \cite[Section 3]{Be2020} are determinantal with repect to some explicit correlation kernels.

Furthermore we establish a number of generalized Littlewood identities summed over certain restricted partitions
by computing the inner products with 
elements indexed by one-column partitions {or generalized partitions $(0^m)$} in the complete dual Fock space. {In particular,
for each positive integer $n$, we obtain generalized Littlewood identities for $(-n)$-asymmetric partitions. We show that
these generalized Littlewood identities contain several well-known Littlewood-type identities as special cases. Consequently we also give a
new proof of the generalized Littlewood identity \cite[(5.25)]{LSV2008} for Lie superalgebras. }


\end{abstract}
\section{introduction}
Schur polynomials are characters of irreducible representations of general linear groups $\mathrm{GL}_{k}(\mathbb{C})$, and are central objects in algebraic combinatorics with various
beautiful properties and numerous applications\cite{Wey1946,FH1991,Me2019}. Dual Littlewood identities (of types B, C, D) are important sum-product identities involving the Schur polynomials\cite{Lit1950}\cite[P.79]{Mac1995}:
\begin{align}
\label{eq:littlewood 0 asym}
\sum_{\la}(-1)^{\frac{|\la|+\rk \la}{2}} s_\la(\x)
&=\prod^k_{i}(1-x_i)\prod_{1 \leq i< j \leq k}(1-x_ix_j)\\
\label{eq:littlewood 1 asym}
\sum_{\la} s_\la(\x)
&=\prod_{1 \leq i\leq j \leq k}(1+x_ix_j) ,\\
\label{eq:littlewood -1 asym}
\sum_{\la} s_\la(\x)
&=\prod_{1 \leq i< j \leq k}(1+x_ix_j),
\end{align}
where the sums run over all partitions $\la=(\alpha|\alpha)$, $\la=(\alpha+1|\alpha)$ and $\la=(\alpha|\alpha+1)$ in the Frobenius notation
corresponding respectively to types $B$, $C$, and $D$. Here $\alpha$ is a strict partition. These are referred to as partitions of type $B, C$ and $D$ respectively. Type $B$ partitions are self-conjugate, type $C$ partitions
are usually known as doubled distinct partitions\cite[Section 8]{GKS1990} {(or $1$-asymmetric partitions)}, while type $D$ partitions are conjugate of doubled distinct partitions {(or $(-1)$-asymmetric partitions)}.

These identities naturally give rise to three measures on partitions,
\begin{align}
\label{eq:Littlewood B measure}
&\mathfrak{M}_{B}(\la)=\frac{(-1)^{\frac{|\la|+\rk\la}{2}} s_\la(\x)}{\prod^k_{i}(1-x_i)\prod_{1 \leq i< j \leq k}(1-x_ix_j)},~~~~&\la=(\alpha|\alpha),\\
\label{eq:Littlewood C measure}
&\mathfrak{M}_{C}(\la)=\frac{s_{\la}(\x)}{\prod_{1 \leq i\leq j \leq k}(1+x_ix_j)},~~~~&\la=(\alpha+1|\alpha),\\
\label{eq:Littlewood D measure}&\mathfrak{M}_{D}(\la)=\frac{s_{\la}(\x)}{\prod_{1 \leq i< j \leq k}(1+x_ix_j)},~~~~&\la=(\alpha|\alpha+1).
\end{align}
In \cite[Section 7]{Rai2000}, Rains used algebraic method to show these measures are determinantal with explicit correlation kernels. Moreover, the
Littlewood identities and refined ones 
have recently been shown to play an important role in integrable probability \cite{BBC2020,BBCW2018,BW2016,BR2005,BZ2019} as well.

From the classical Cauchy identity (of type $A$)
\begin{align*}
\sum_{\la}s_{\la}(\x)s_{\la}(\y)=\prod^k_{i=1}\prod^n_{j=1}\frac{1}{1-x_iy_j}
\end{align*}
and with the free fermion (vertex operator) realization, Okounkov\cite{Oko2001} introduced the Schur measure
\begin{align*}
\mathfrak{M}_{\la}=\prod^k_{i=1}\prod^n_{j=1}(1-x_iy_j)s_{\la}(\x)s_{\la}(\y)
\end{align*}
and proved that the Schur measure is a determinantal one with an explicit
correlation kernel. In the past twenty years, the Schur measure has provided rich
interplay between combinatorics, probability, statistical mechanics, and enumerative geometry. Using the Cauchy-type identity for the universal symplectic characters, Cuenca and Mucciconi defined the symplectic Schur process \cite{CM2024}. Very recently, Betea, Nazarov, Nikitin and Scrimshaw \cite{BNNS2024} studied the dual version of the Schur measure defined by
\begin{align*}
\mu_{k,n}(\la)=\frac{s_{\la}(\x)s_{\la^{\prime}}(\y)}{\prod^k_{i=1}\prod^n_{j=1}(1+x_iy_j)}
\end{align*} based on the dual Cauchy identity
\begin{align*}
\sum_{\la}s_{\la}(\x)s_{\la^{\prime}}(\y)=\prod^k_{i=1}\prod^n_{j=1}(1+x_iy_j).
\end{align*}

For the Cauchy identities of types $B, C$, and $D$, Betea \cite{Be2020} defined the symplectic and orthogonal Schur measures and found that each measure is a determinantal point process with an explicit
correlation kernel in analogy with Okounkov’s Schur measure by using the vertex operator method.  Betea also pointed out that the measure $\mathfrak{M}_{C}(\la)$ (resp., $\mathfrak{M}_{D}(\la)$) is a special case of symplectic (resp., orthogonal) Schur measures, and also specializes to the Poissonized Plancherel measure \cite{Be2020}.
{\it The first goal of this paper is to study the measures \eqref{eq:Littlewood B measure}, \eqref{eq:Littlewood C measure} and \eqref{eq:Littlewood D measure} using a new free fermionic (or vertex operator) construction and show that they are determinantal}.

The Frobenius notation for partitions show that there exist a one-to-one correspondence between the set of partitions for measure $\mathfrak{M}_{C}(\la)$ and
the set of strict partitions, for example $(\alpha+1|\alpha)\longrightarrow \alpha$. Though the latter is a subset of $\mathcal P$, it is well-known that many properties of Schur functions can be generalized
to Schur's Q-functions. It is interesting to note that the set of partitions for the measures related to other classical types also display similar phenomena.

The main methodology of this paper is vertex calculus. The classical Boson-Fermion correspondence (vertex operator representation) shows that there exists an isomorphism between two level one representations of the infinite-dimensional Heisenberg algebra $\mathcal{H}$: the space $\Lambda$ of symmetric polynomials and the Fock space $\mathcal{M}$. It identifies the Schur polynomial $s_\la(\x)$ with the partition element $|\la\rangle$ in $\mathcal{M}$ \cite{DJKM1983,FK1980,KRR2013,Jing1991,JR2016}. In parallel, the correspondence between symplectic Schur polynomial $sp_\la(\x^{\pm})$ (resp. orthogonal Schur polynomial $o_\la(\x^{\pm})$) and $|\la^{sp}\rangle$ (resp. $|\la^{o}\rangle$) was found in \cite{Ba1996,JN2015}. Using vertex operator representations, an identity between the odd orthogonal Schur polynomials $so_\la(\x^{\pm})$ and $|\la^{so}\rangle$ was also established\cite{JLPWY2025}. In \cite{JLW2024, JLPWY2025}, for generalized partitions $\la=(\la_1,\dots,\la_l)$ and $\mu=(\mu_1,\dots,\mu_l)$, 
the third and fourth named
authors together with collaborators have showed the following three orthonormal relations
\begin{align}\label{ortho-relation}
&\langle \mu^{so}|\la^{so}\rangle=\delta_{\la\mu},\\
&\langle \mu^{sp}|\la^{sp}\rangle=\delta_{\la\mu},\\
&\langle \mu^{o}|\la^{o}\rangle=\delta_{\la\mu},
\end{align}
where $\langle \mu^{so}|$, $\langle \mu^{sp}|$ and $\langle \mu^{o}|$ are partition elements in the completed dual Fock space $\widetilde{\mathcal{M}}^*$.

In \cite{JL2022}, these two authors used a single vertex operator to study the dual Cauchy identity.  In this paper, we introduce a family of new partition elements $|\underline{\la}^{sp}\rangle$ (resp. $|\underline{\la}^{o}\rangle$, $|\underline{\la}^{so}\rangle$) and calculate their inner products with the dual vacuum $\langle0|$.
We show that nonzero inner products $\langle0|\underline{\la}^{sp}\rangle$ lead to a new proof of the dual Littlewood identity for type C as given in \eqref{eq:littlewood 1 asym}. The remaining two dual Littlewood identities can be proved similarly. This gives rise to vertex operator constructions of the measures $\mathfrak{M}_{B}(\la)$, $\mathfrak{M}_{C}(\la)$ and $\mathfrak{M}_{D}(\la)$. Our {\it first} result is an integral representation for the correlation kernels.
In the following we will use $\mathbf{i}=\sqrt{-1}$ and $i$ as an indexing variable.
\begin{theorem}(i.e. Theorem \ref{thm:CD-measures} and Theorem \ref{thm:B-measure})
For $\la=(\alpha+1|\alpha)$, $\la=(\alpha|\alpha+1)$ or $\la=(\alpha|\alpha)$, let $n_i=\la_i-i+1$. The measures $\mathfrak{M}_{C}(\la)$, $\mathfrak{M}_{D}(\la)$ and $\mathfrak{M}_{B}(\la)$ are determinantal ensembles
\begin{align}
&\mathfrak{M}_{C}(\la)=\det\big(K_C(n_i,n_j)\big)_{1\leq i,j\leq k},~~~~~~~~\la=(\alpha+1|\alpha),\\
&\mathfrak{M}_{D}(\la)=\det\big(K_D(n_i,n_j)\big)_{1\leq i,j\leq k},~~~~~~~~\la=(\alpha|\alpha+1),\\
&\mathfrak{M}_{B}(\la)=\det\big(K_B(n_i,n_j)\big)_{1\leq i,j\leq k},~~~~~~~~\la=(\alpha|\alpha),
\end{align}
where the correlation kernels are given by the following integrals: 
\begin{align}\label{eq:C-type kernel}
&K_C(a,b)=\int_{z}\int_{w}\frac{1-z^2}{(1-wz)(1-wz^{-1})}\frac{H(z)}{H(w)}\frac{dzdw}{(2\pi \mathbf{i})^2z^{b+1}w^{-a+1}},\\
&K_D(a,b)=\int_{z}\int_{w}\frac{1-w^2}{(1-wz)(1-wz^{-1})}\frac{H(z)}{H(w)}\frac{dzdw}{(2\pi \mathbf{i})^2z^{b+1}w^{-a+1}},\\
&K_B(a,b)=\int_{z}\int_{w}\frac{(1+z)(1-w)}{(1-wz)(1-wz^{-1})}\frac{J(z)}{J(w)}\frac{dzdw}{(2\pi \mathbf{i})^2z^{b+1}w^{-a+1}}
\end{align}
with
\begin{align*}
&H(z)=\prod^k_{i=1}(1-\mathbf{i}x_iz)(1-\mathbf{i}x_iz^{-1}),\\
&J(z)=\prod^k_{i=1}(1-x_iz)(1-x_iz^{-1})
\end{align*}
and the $z-$ and $w-$ contours are simple counterclockwise circles around 0 satisfying $|w|<|z|<|x_i|$.
\end{theorem}

The generalized family of partitions also offers some nontrivial combinatorial identities. Our {\it second} result (Theorem \ref{thm:Littlewood-column}) derives infinitely many generalized Littlewood identities, where the classical ones have played important role in studying symmetric functions (cf. \cite{Mac1995}). We offer the nonzero conditions of the inner products $\langle \mu^{sp}|\underline{\la}^{sp}\rangle$, $\langle \mu^{o}|\underline{\la}^{o}\rangle$ and $\langle \mu^{so}|\underline{\la}^{so}\rangle$ for one-column partitions $\mu$. As a consequence, we obtain three
new families of identities generalizing the classical identities \eqref{eq:littlewood 0 asym}, \eqref{eq:littlewood 1 asym} and \eqref{eq:littlewood -1 asym}. Generalized Littlewood identity \eqref{eq:littlewood-column D} are closely related with the recently emerged bounded Littlewood identities
given by Erickson and Hunziker \cite[Table 4]{EH2026} and Huh, Kim, Krattenthaler and Okada \cite[(5.3)]{HKKO2025b}. For any fixed partition $\mu$, our
 approach suggests that the nonzero condition of $\langle\mu^{sp}|\underline{\la}^{sp}\rangle$ for partitions $\la$ will produce a new generalized Littlewood identity.

In fact, the generalized Littlewood identities considered by Erickson and Hunziker \cite{EH2026} for the so-called $(-n)$-asymmetric partitions
were obtained in the context of the BGG resulotions. Correspondingly for any positive integer $n$, we also obtain the
following new generalized Littlewood identities for $(-n)$-asymmetric partitions, as the {\it third} result of the paper.
\begin{small}
\begin{align*}
&\sum_{\la=(\alpha+2m|\alpha)}(-1)^{\frac{|\la|+\rk\la}{2}+m\rk\la} s_\la(\x)=\prod^k_{i}(1-x_i)\prod_{1 \leq i< j \leq k}(1-x_ix_j)\det_{1\leq s,t\leq m}\left(f_{s-t}(\x)+f_{s+t-1}(\x)\right)  ,\\
&\sum_{\la=(\alpha+2m+1|\alpha)}(-1)^{\frac{|\la|}{2}+m\rk\la} s_\la(\x)=\prod_{1 \leq i\leq j \leq k}(1-x_ix_j)\det_{1\leq s,t\leq m}\left(f_{s-t}(\x)-f_{s+t}(\x)\right) ,\\
&\sum_{\la=(\alpha+2m-1|\alpha)}(-1)^{\frac{|\la|}{2}+m\rk\la} s_\la(\x)=\frac{1}{2}\prod_{1 \leq i< j \leq k}(1-x_ix_j)\det_{1\leq s,t\leq m}\left(f_{s-t}(\x)+f_{s+t-2}(\x)\right),
\end{align*}
\end{small}
where
\begin{align*}
f_i(\x)=f_{-i}(\x)=\sum^{k-i}_{j=0}e_j(\x)e_{i+j}(\x).
\end{align*}
Combining with \cite[Theorem 1.1]{HKKO2025b} and \cite[Theorem 7.1]{Ste1990}, we obtain the following identities
\begin{align}\label{e:char1}
&\prod^k_{i}(1-x_i)\prod_{1 \leq i< j \leq k}(1-x_ix_j)\sum_{\lambda_1\leq 2m+1}s_{\la}(\x)=\sum_{\la=(\alpha+2m+1|\alpha)}(-1)^{\frac{|\la|}{2}+m\rk\la} s_\la(\x),\\
\label{e:char2}&\prod^k_{i}(1-x_i)\prod_{1 \leq i< j \leq k}(1-x_ix_j)\sum_{\lambda_1\leq 2m}s_{\la}(\x)=\sum_{\la=(\alpha+2m|\alpha)}(-1)^{\frac{|\la|+\rk\la}{2}+m\rk\la} s_\la(\x).
\end{align}
Under the involution $\omega$ sending $s_{\la}(\x)$ to $s_{\la'}(\x)$, our \eqref{e:char1} and \eqref{e:char2} are actually
the following identity:
\begin{align}
\sum_{l(\lambda)\leq n}s_{\la}(\x)=\frac{\sum_{\la=(\alpha|\alpha+n)}(-1)^{|\alpha|+\rk\la} s_\la(\x)}{\prod^k_{i}(1-x_i)\prod_{1 \leq i< j \leq k}(1-x_ix_j)},
\end{align}
which was conjectured by Lievens, Stoilova, Van der Jeugt\cite[(5.25)]{LSV2008} and proved by King\cite{Kin2013}.

The paper is organized as follows. We recall all the definitions and preliminary results in section 2. In section 3, we revisit the classical dual Littlewood identities using vertex operators. In section 4, we study the measures resulting from the dual Littlewood identities. Then we derive the infinitely many generalized Littlewood identities in section 5. {
In section 6, generalized Littlewood identities for $(-n)$-asymmetric partitions are also obtained. We then study the connections between our generalized Littlewood identities and some of the known Littlewood-type identities in Section 7.}

\section{Preliminaries}
In this section, we recall some definitions and results of the Heisenberg algebra, generalized partitions, vertex operators and symmetric polynomials. We
mostly follow the notations in \cite{JLW2024,JLPWY2025,Mac1995}.
\subsection{Heisenberg algebra}Let $\mathcal{H}$ be the Heisenberg algebra generated by $\{a_n\}_{n\in\mathbb{Z}\setminus\{0\}}$ with central element $c=1$, subject to the commutation relations \cite{FK1980}:
\begin{align}\label{e:he1}
[a_m,a_n]=m\delta_{m,-n}c,~~~~[a_n,c]=0.
\end{align}
The Fock space $\mathcal{M}$ (resp. $\mathcal{M}^*$) is generated by the vacuum vector $|0\rangle$ (resp. dual vacuum vector $\langle0|$) and
subject to
\begin{align*}
a_n|0\rangle=\langle0|a_{-n},~~n>0.
\end{align*}
Let $\mathcal{M}_n^*$ be the subspace of degree $n$ spanned by $\langle 0|a_{-i_1}\cdots a_{-i_n}$, then $\mathcal{M}^*=\oplus_{n=0}^{\infty}\mathcal{M}_n^*$. Let $\overline{\mathcal{M}}^*_n$ be the graded space $\oplus_{i\leq n} {\mathcal{M}}^*_i$.
The completion $\widetilde{\mathcal{M}}^*$ is the inverse limit of $\overline{\mathcal{M}}^*_n$.

\subsection{Generalized partitions}
A {\it generalized partition} $\la=(\la_1,\la_2,\cdots,\la_l)$ of weight $|\la|=\sum_i\la_i$
is a finite sequence of weakly decreasing nonnegative integers such that the total sum is $|\la|$. $\la_i$ are called parts of $\la$, and the length $l(\la)$ of $\la$ is the number of nonzero parts. A {\it partition} is a sequence of weakly decreasing positive integers\footnote{A partition is a special case of generalized partitions. We distinguish partitions from generalized partitions.}. The set of all partitions is denoted
by $\mathcal{P}$. The {\it conjugate} of a partition $\la$ is the partition $\la^{\prime}$ with
\begin{align*}
\la^{\prime}_i=\text{Card}\{j:\la_j\geq i\}.
\end{align*}

{
The {\textit{Young diagram}} of a partition $(\lambda=(\la_1,\la_2,\dots, \la_k) )$
is the set $( \{ (i,j) \in \Z^2: 1\le i\le k, 1\le j\le \lambda_i\} )$. Each element $
(i,j)$ in a Young diagram is called a \textit{cell}.
We often identify a partition with its Young diagram and write $\Par(p \times q)$ for the set of partitions whose Young diagrams fit inside a rectangle with $p$ rows and $q$ columns. }
For a partition $\la$, let $\alpha_i=\la_i-i$ and $\beta_i=\la^{\prime}_i-i$, both $\alpha, \beta$ are strict. The (Frobenius) rank of a partition
$\la$, denoted $\rk\la$, is the largest integer $k$ such that $\la_k\geq k$. The {\it Frobenius notation} for the partition $\lambda$ is given by
$$\la=(\alpha|\beta)=(\alpha_1,\dots,\alpha_r|\beta_1,\dots,\beta_r),$$
where $r=\rk\la$. {
In particular, we say a partition $\lambda$ is {\it $(-n)$-asymmetric} if $\lambda=(\alpha+n|\alpha)$\cite{Alb2025}\cite{AK2022}\cite{JW2025}.}
Here $\alpha+n$ means $(\alpha_1+n, \alpha_2+n, \ldots, \alpha_r+n)$.
We will sometimes
write $(n^m)$ for the rectangular partition with $m$ parts equal to $n$.

\subsection{Vertex operators}Define the following vertex operators\cite{Ba1996,JLW2024,JLPWY2025,JN2015}
\begin{equation}
\begin{aligned}\label{ver-BCD}
&U(z)=(1+z)\exp\left(\sum^\infty_{n=1}\frac{a_{-n}}{n}z^n\right)\exp\left(-\sum^\infty_{n=1}\frac{a_n}{n}(z^{-n}+z^n)\right)=\sum_{n\in \mathbb{Z}}U_nz^{-n},\\
&U^*(z)=(1-z)\exp\left(-\sum^\infty_{n=1}\frac{a_{-n}}{n}z^n\right)\exp\left(\sum^\infty_{n=1}\frac{a_n}{n}(z^{-n}+z^n)\right)=\sum_{n\in \mathbb{Z}}U^*_nz^n,\\
&Y(z)=\exp\left(\sum^\infty_{n=1}\frac{a_{-n}}{n}z^n\right)\exp\left(-\sum^\infty_{n=1}\frac{a_n}{n}(z^{-n}+z^n)\right)=\sum_{n\in \mathbb{Z}}Y_nz^{-n},\\
&Y^*(z)=(1-z^2)\exp\left(-\sum^\infty_{n=1}\frac{a_{-n}}{n}z^n\right)\exp\left(\sum^\infty_{n=1}\frac{a_n}{n}(z^{-n}+z^n)\right)=\sum_{n\in \mathbb{Z}}Y^*_nz^n,\\
&W(z)=(1-z^2)\exp\left(\sum^\infty_{n=1}\frac{a_{-n}}{n}z^n\right)\exp\left(-\sum^\infty_{n=1}\frac{a_n}{n}(z^{-n}+z^n)\right)=\sum_{n\in \mathbb{Z}}W_nz^{-n},\\
&W^*(z)=\exp\left(-\sum^\infty_{n=1}\frac{a_{-n}}{n}z^n\right)\exp\left(\sum^\infty_{n=1}\frac{a_n}{n}(z^{-n}+z^n)\right)=\sum_{n\in \mathbb{Z}}W^*_nz^n.
\end{aligned}
\end{equation}
It is easy to check that
\begin{equation}
\begin{aligned}\label{vacuum-action}
&U_n|0\rangle=U^*_{-n}|0\rangle=0,~~n>0,\quad\langle 0|U^*_n=-\langle 0|U^*_{-n+1},~~\langle 0|U_n=\langle 0|U_{-n-1},\\
&Y_n|0\rangle=Y^*_{-n}|0\rangle=0,~~n>0,\quad \langle 0|Y^*_n=-\langle 0|Y^*_{-n+2},~~ \langle 0|Y_n=\langle 0|Y_{-n},\\
&W_n|0\rangle=W^*_{-n}|0\rangle=0,~~n>0,\quad  \langle 0|W^*_n=\langle 0|W^*_{-n},~~\langle 0|W_n=-\langle 0|W_{-n-2},\\
&U_0|0\rangle=U^*_0|0\rangle=Y_0|0\rangle=Y^*_0|0\rangle=W_0|0\rangle=W^*_0|0\rangle=|0\rangle.
\end{aligned}
\end{equation}

The following results follow from vertex algebraic techniques \cite[Theorem 3.4]{JN2015} and
the Baker-Campbell-Hausdorff formula
\begin{align}\label{e:BCH}
\exp(A)\exp(B)=\exp(B)\exp(A)\exp([A,B]),~~~~~~~~~~~[[A,B],A]=[[A,B],B]=0.
\end{align}
\begin{proposition}[\cite{JN2015,JLW2024,JLPWY2025}]
The operators $U_i,U^*_i,Y_i,Y^*_i,W_i,W^*_i$ act linearly on $\mathcal M$ and $U^*_i,Y^*_i, W^*_i$ act linearly on $\widetilde{\mathcal M}^*$, and
satisfy the following commutation relations:
\begin{equation}
\begin{aligned}\label{anti-com}
U_iU_j+U_{j+1}U_{i-1}=&0,\qquad U^*_iU^*_j+U^*_{j-1}U^*_{i+1}=0,\qquad U_iU^*_j+U^*_{j+1}U_{i+1}=\delta_{i,j},\\
Y_iY_j+Y_{j+1}Y_{i-1}=&0,\qquad Y^*_iY^*_j+Y^*_{j-1}Y^*_{i+1}=0,\qquad Y_iY^*_j+Y^*_{j+1}Y_{i+1}=\delta_{i,j},\\
W_iW_j+W_{j+1}W_{i-1}=&0,\qquad W^*_iW^*_j+W^*_{j-1}W^*_{i+1}=0,\qquad W_iW^*_j+W^*_{j+1}W_{i+1}=\delta_{i,j}.
\end{aligned}
\end{equation}
\end{proposition}


For a generalized partition $\la=(\la_1,\la_2,\dots,\la_l)$, we denote the following partition vectors:
\begin{equation}
\begin{aligned}\label{par-element}
&|\la^{so}\rangle=U_{-\la_1}\cdots U_{-\la_l}|0\rangle, \qquad |\underline{\la}^{so}\rangle=U^{*}_{\la_1}\cdots U^{*}_{\la_k}|0\rangle, \qquad\langle \la^{so}|=\langle 0|U^*_{-\la_l}\cdots U^*_{-\la_1},\\
&|\la^{sp}\rangle=Y_{-\la_1}\cdots Y_{-\la_l}|0\rangle, \qquad |\underline{\la}^{sp}\rangle=Y^{*}_{\la_1}\cdots Y^{*}_{\la_k}|0\rangle,\qquad\langle \la^{sp}|=\langle 0|Y^*_{-\la_l}\cdots Y^*_{-\la_1},\\
&|\la^{o}\rangle=W_{-\la_1}\cdots W_{-\la_l}|0\rangle, \qquad |\underline{\la}^{o}\rangle=W^{*}_{\la_1}\cdots W^{*}_{\la_k}|0\rangle,\qquad\langle \la^{o}|=\langle 0|W^*_{-\la_l}\cdots W^*_{-\la_1},
\end{aligned}
\end{equation}
where the elements $\langle \la^{so}|, \langle \la^{sp}|, \langle \la^{o}|\in \widetilde{\mathcal M^*}$.
\begin{lemma}[\cite{JLW2024,JLPWY2025}]For generalized partitions $\la=(\la_1,\dots,\la_l)$ and $\mu=(\mu_1,\dots,\mu_l)$, the partition elements are orthonormal vectors
in the following sense: 
\begin{align}
\label{ortho-relationB}
&\langle \la^{so}|\mu^{so}\rangle=\delta_{\la\mu},\\
\label{ortho-relationC}
&\langle \la^{sp}|\mu^{sp}\rangle=\delta_{\la\mu},\\
\label{ortho-relationD}
&\langle \la^{o}|\mu^{o}\rangle=\delta_{\la\mu}.
\end{align}
\end{lemma}

\begin{lemma}[\cite{JLW2024,JLPWY2025}]For any generalized partition $\mu=(\mu_1,\dots,\mu_l)$, one has that for any $\sigma\in \mathfrak S_l$
\begin{align}
\label{permutation-B}
&\varepsilon(\sigma)\langle 0|\left(^{U^*_{-\mu_{\sigma(l)}+\sigma(l)-l}}_{-U^*_{\mu_{\sigma(l)}-\sigma(l)+l+1}}\right)\cdots \left(^{U^*_{-\mu_{\sigma(i)}+\sigma(i)-i}}_{-U^*_{\mu_{\sigma(i)}-\sigma(i)+2l+1-i}}\right)\cdots \left(^{U^*_{-\mu_{\sigma(1)}+\sigma(1)-1}}_{-U^*_{\mu_{\sigma(1)}-\sigma(1)+2l}}\right)=\langle\mu^{so}|,\\
\label{permutation-C}
&\varepsilon(\sigma)\langle 0|\left(^{Y^*_{-\mu_{\sigma(l)}+\sigma(l)-l}}_{-Y^*_{\mu_{\sigma(l)}-\sigma(l)+l+2}}\right)\cdots \left(^{Y^*_{-\mu_{\sigma(i)}+\sigma(i)-i}}_{-Y^*_{\mu_{\sigma(i)}-\sigma(i)+2(l+1)-i}}\right)\cdots \left(^{Y^*_{-\mu_{\sigma(1)}+\sigma(1)-1}}_{-Y^*_{\mu_{\sigma(1)}-\sigma(1)+2l+1}}\right)=\langle\mu^{sp}|,\\
\label{permutation-D}
&\varepsilon(\sigma)\langle 0|\left(^{W^*_{-\mu_{\sigma(l)}+\sigma(l)-l}}_{\delta_lW^*_{\mu_{\sigma(l)}-\sigma(l)+l}}\right)\cdots \left(^{W^*_{-\mu_{\sigma(i)}+\sigma(i)-i}}_{\delta_iW^*_{\mu_{\sigma(i)}-\sigma(i)+2l-i}}\right)\cdots \left(^{W^*_{-\mu_{\sigma(1)}+\sigma(1)-1}}_{\delta_1W^*_{\mu_{\sigma(1)}-\sigma(1)+2l-1}}\right)=\langle\mu^{o}|,
\end{align}
where $\delta_i$ denotes $\delta_{\sigma(i)\neq l}\delta_{\mu_l\neq 0}$\footnote{$\delta_{a\neq b}$ equals 0 for $a=b$, 1 otherwise. } and the symbol $ \left(^a_{b}\right)$ means either $a$ or $b$.
\end{lemma}
\subsection{Symmetric polynomials} The readers are referred to the basic references \cite[Chapter 1]{Mac1995} and \cite[Chapter 7]{Sta1999} for the general background of symmetric polynomials. We denote the algebra of symmetric polynomials by $\Lambda$. For
the variable $\x=(x_1,\dots,x_k)$, the elementary symmetric polynomials $e_n(\x)$ and the complete symmetric polynomials $h_n(\x)$ are respectively defined by
\begin{align*}
&\sum_{n\in \mathbb{Z}}e_n(\x)z^n=\prod^k_{i=1}(1+x_iz),\\
&\sum_{n\in \mathbb{Z}}h_n(\x)z^n=\prod^k_{i=1}\frac{1}{1-x_iz}.
\end{align*}
The Schur polynomials $s_{\la}(\x)$ are defined as
the bialternant and determinants:
\begin{align*}
&s_{\la}(\x)=\frac{\det\left(x^{\la_j+k-j}_i\right)_{1\leq i,j\leq k}}{\det\left(x^{k-j}_i\right)_{1\leq i,j\leq k}}, \text{or}\\
&s_{\la}(\x)=\det\left(h_{\la_i-i+j}(\x)\right)_{1\leq i,j\leq l(\lambda)}=\det\left(e_{\la^{\prime}_i-i+j}(\x)\right)_{1\leq i,j\leq \lambda_1}
\end{align*}
for $l(\la)\leq k.$ It is well known that $s_{\la}(\x)$ is the character of the irreducible representation of $\mathrm{GL}_k(\mathbb C)$ indexed by the highest weight $\la$.
\section{dual Littlewood identities} In the section, we revisit dual Littlewood identities using vertex operator methods.
\begin{lemma}\label{le:classical Littlewood}
For a partition $\la=(\la_1,\dots,\la_k)$, we have
\begin{equation}
\label{vacuum-parB}
\langle 0|\underline{\la}^{so}\rangle=
\begin{cases}
(-1)^{\frac{|\la|+\rk\la}{2}} & \la=(\alpha|\alpha)\\
0 & \text{otherwise}
\end{cases},
\end{equation}
\begin{equation}
\label{vacuum-parC}
\langle 0|\underline{\la}^{sp}\rangle=
\begin{cases}
(-1)^{\frac{|\la|}{2}} & \la=(\alpha+1|\alpha)\\
0 & \text{otherwise}
\end{cases},
\end{equation}
\begin{equation}
\label{vacuum-parD}
\langle 0|\underline{\la}^{o}\rangle=
\begin{cases}
(-1)^{\frac{|\la|}{2}} & \la=(\alpha|\alpha+1)\\
0 & \text{otherwise}
\end{cases}.
\end{equation}
\end{lemma}
\begin{proof}
We only prove the relation \eqref{vacuum-parC}, as the other two relations can be treated similarly.
The orthogonality relation \eqref{ortho-relationC} shows that the inner product
\begin{align}
\langle 0|\underline{\la}^{sp}\rangle=\langle 0|Y^{*}_{\la_1}Y^{*}_{\la_2}\cdots Y^{*}_{\la_k}|0\rangle
\end{align}
is nonzero only for the case
\begin{align}\label{nonzero-C}
\langle 0|Y^{*}_{\la_1}Y^{*}_{\la_2}\cdots Y^{*}_{\la_k}=\varepsilon\langle 0|(Y^{*}_0)^k 
\end{align}
where $\varepsilon\in \{1,-1\}$. It follows from \eqref{permutation-C} and \eqref{nonzero-C} that for some $\sigma\in\mathfrak S_k$, and any integer $j\in [1, k]$
\begin{align}
\la_j=\sigma(k+1-j)-(k+1-j) ~~\text{or}~~-\sigma(k+1-j)+k+1+j,
\end{align}
i.e.,
\begin{align}
\varepsilon_j(\la_j-j)=k+1-\sigma(k+1-j)
\end{align}
for $\varepsilon_j\in \{1,-1\}$.
Since $1\leq \sigma(k+1-j)\leq k$, $\varepsilon_j(\lambda_j-j)$ runs through $1, 2, \ldots, k$ as $j$ goes through $1, 2, \ldots, k$. Thus there is a one-to-one correspondence between the two sets
\begin{align}\label{e:corresp}
\{\varepsilon_1(\la_1-1),\varepsilon_2(\la_2-2),\cdots,\varepsilon_k(\la_k-k)\}\longleftrightarrow \{1,2,\cdots,k\}.
\end{align}
It is easily seen that $\la_1\leq k+1$, then $\lambda_i-i\leq k-i+1$. Note that $\lambda_i\geq 1$, so $\lambda_1=k+1$
in view of \eqref{e:corresp}. As a result
\begin{align*}
\{\varepsilon_2(\la_2-2),\cdots,\varepsilon_k(\la_k-k)\}\longleftrightarrow \{1,2,\cdots,k-1\}.
\end{align*}
Since $\lambda_i\leq k+1$, we have $\lambda_2-2=k-1$ if $\lambda_2=k+1$. Continuing this way, we assume that the subscript $s_1$ of $\la_{s_1}$ is the first number such that $\la_{s_1}<k+1$ (and $\la_{i}=k+1$ for $i<s_1$). Then we have
 \begin{align}
\label{e:sp3}\{\varepsilon_{s_1}(\la_{s_1}-s_1),\cdots,\varepsilon_k(\la_k-k)\}\longleftrightarrow \{1,2,\cdots,k+1-s_1\},
\end{align}
as $\la_i-i=k+1-i$ for $1\leq i\leq s_1-1$. Note that $1-k\leq \la_k-k < 0$, we have $\varepsilon_k=-1$ and
\begin{align*}
-(\la_k-k)\leq k+1-s_1
\end{align*}
from \eqref{e:sp3}. Thus $$\la_k\geq s_1-1.$$ Note that $-(\la_{k-1}-(k-1))\leq k-s_1$, then
\begin{align}
\la_k=s_1-1.
\end{align}
Therefore, we have that
\begin{align}
\la^{\prime}_i+1=\la_i,~~~~~~1\leq i\leq s_1-1.
\end{align}

   Let $t_1$ be the minimum subscript of $\la_{k-t_1}$ such that $\la_{k-t_1}>s_1-1$. Then
\begin{align*}
\{\varepsilon_{s_1}(\la_{s_1}-s_1),\cdots,\varepsilon_{k-t_1}(\la_{k-t_1}-(k-t_1))\}\longleftrightarrow \{1,2,\cdots,k+1-s_1-t_1\}.
\end{align*}
Similar to the above procedure, we have $\la_{s_1}-s_1=k+1-s_1-t_1$, i.e.,
\begin{align}
\la_{s_1}=k+1-t_1.
\end{align}
Assuming the subscript $s_2$ of $\la_{s_1+s_2}$ is the first number such that $\la_{s_1+s_2}<k+1-t_1$, we have
 \begin{align}
\label{e:sp4}\{\varepsilon_{s_1+s_2}(\la_{s_1+s_2}-(s_1+s_2)),\cdots,\varepsilon_{k-t_1}(\la_{k-t_1}-(k-t_1))\}\longleftrightarrow \{1,2,\dots,k+1-s_1-t_1-s_2\}.
\end{align}
As above, $\la_{k-t_1}=s_1+s_2-1$. Thus we have
\begin{align}
\la^{\prime}_i+1=\la_i,~~~~~~s_1\leq i\leq s_1+s_2-1.
\end{align}
Continuing the procedure until $\la_s-s=1$, i.e., $\la_s=s+1$, we see that $\la$ is of the form
\begin{align}
\la=(\alpha_1+1,\alpha_2+1,\cdots|\alpha_1,\alpha_2,\cdots)
\end{align}
with $\la_1\leq k+1$.

For $\la=(\alpha_1+1,\alpha_2+1,\cdots|\alpha_1,\alpha_2,\cdots)$,
\begin{align}\label{sign-C}
\notag\langle 0|\underline{\la}^{sp}\rangle=&\langle 0|Y^{*}_{\la_1}Y^{*}_{\la_2}\cdots Y^{*}_{\la_k}|0\rangle\\
\notag=&-\langle 0|Y^{*}_{-\la_1+2}Y^{*}_{\la_2}\cdots Y^{*}_{\la_k}|0\rangle\\
\notag=&(-1)^{k}\langle 0|Y^{*}_{\la_2-1}\cdots Y^{*}_{\la_k-1}Y^{*}_{0}|0\rangle\\
=&(-1)^{k}\langle 0|Y^{*}_{\la_2-1}\cdots Y^{*}_{\la_k-1}|0\rangle,
\end{align}
where the second and third equations have used \eqref{vacuum-action}, \eqref{anti-com} and the fact $k=\la_1-1$. Note that $\nu=(\la_2-1,\dots,\la_k-1)$ is also a generalized partition and $k=\frac{|\la|-|\nu|}{2}$, thus
\eqref{sign-C} becomes
\begin{align}
\langle 0|\underline{\la}^{sp}\rangle=(-1)^{\frac{|\la|-|\nu|}{2}}\langle 0|\underline{\nu}^{sp}\rangle.
\end{align}
Continuing the process, we can complete the proof.
\end{proof}
\begin{theorem}\label{thm:Littlewood-o}For $\x=(x_1,\dots,x_k)$, one has that
\begin{align}
\label{eq:littlewood B}
\sum_{\la=(\alpha|\alpha)}(-1)^{\frac{|\la|+\rk\la}{2}} s_\la(\x)
&=\prod^k_{i}(1-x_i)\prod_{1 \leq i< j \leq k}(1-x_ix_j) ,\\
\label{eq:littlewood C}
\sum_{\la=(\alpha+1|\alpha)}(-1)^{\frac{|\la|}{2}} s_\la(\x)
&=\prod_{1 \leq i\leq j \leq k}(1-x_ix_j) ,\\
\label{eq:littlewood D}
\sum_{\la=(\alpha|\alpha+1)}(-1)^{\frac{|\la|}{2}} s_\la(\x)
&=\prod_{1 \leq i< j \leq k}(1-x_ix_j),
\end{align}
\end{theorem}
\begin{proof}
Following \cite{JLW2024}, we introduce the following half vertex operators 
\begin{align}
\notag&\Gamma_-(z)=\exp\left(\sum^\infty_{n=1}\frac{a_{-n}}{n}z^n\right), ~~~~\Gamma_+(z)=\exp\left(\sum^\infty_{n=1}\frac{a_{n}}{n}z^n\right),\\
\label{eq:half-vertex}&\Gamma_-(\x)=\prod^k_{i=1}\Gamma_-(x_i),~~~~~~~~~~~~~~~\Gamma_+(\x)=\prod^k_{i=1}\Gamma_+(x_i),\\
\notag&\Gamma_C(\x)=\prod^k_{i=1}(1-x^2_i)\Gamma^{-1}_-(x_i)\Gamma_+(x_i).
\end{align}
Using the Baker-Campbell-Hausdorff formula \eqref{e:BCH}, we can rewrite
\begin{align}
\Gamma_C(\x)=\prod_{1\leq i\leq j\leq k}(1-x_ix_j)\Gamma^{-1}_-(\x)\Gamma_+(\x).
\end{align}
Note that $\Gamma_+(\x)|0\rangle=|0\rangle$ and $\langle0|\Gamma^{-1}_-(\x)=\langle0|$, we immediately have that
\begin{equation}\label{LittlewoodC-1}
\begin{aligned}
&\langle0|\Gamma_C(\x)|0\rangle\\
~~~~&=\prod_{1\leq i\leq j\leq k}(1-x_ix_j)\langle0|\Gamma^{-1}_-(\x)\Gamma_+(\x)|0\rangle\\
~~~~&=\prod_{1\leq i\leq j\leq k}(1-x_ix_j).
\end{aligned}
\end{equation}
From the definition of the vertex operator $Y^{*}(z)$, we have
\begin{align}\label{eq:LittlewoodC-generating}
\notag\Gamma_C(\x)|0\rangle=&Y^{*}(x_1)\Gamma^{-1}_+(x^{-1}_1)\cdots Y^{*}(x_k)\Gamma^{-1}_+(x^{-1}_k)|0\rangle\\
\notag=&\prod_{i<j}\frac{1}{1-\frac{x_j}{x_i}}Y^{*}(x_1)\cdots Y^{*}(x_k)|0\rangle\\
\notag=&\prod_{i<j}\frac{1}{1-\frac{x_j}{x_i}}\sum_\sigma\prod_ix^{\lambda_{\sigma(i)}-\sigma(i)+i}_i|\underline{\la}^{sp}\rangle\\
\notag=&\prod_{i<j}\frac{1}{x_i-x_j}\sum_\sigma\prod_ix^{\lambda_{\sigma(i)}-\sigma(i)+k}_i|\underline{\la}^{sp}\rangle\\
=&\sum_{\la}s_\la(\x)|\underline{\la}^{sp}\rangle,
\end{align}
and thus have
\begin{align}\label{LittlewoodC-2}
\langle0|\Gamma_C(\x)|0\rangle=\sum_{\la=(\alpha+1|\alpha)}(-1)^{\frac{|\la|}{2}}s_\la(\x).
\end{align}
 by \eqref{vacuum-parC}. Comparing \eqref{LittlewoodC-1} and \eqref{LittlewoodC-2}, we can prove \eqref{eq:littlewood C}.

  Similarly, let
  \begin{align}
\label{eq:half-vertexB}  &\Gamma_B(\x)=\prod^k_{i=1}(1-x_i)\Gamma^{-1}_-(x_i)\Gamma_+(x_i),\\
\label{eq:half-vertexD}  &\Gamma_D(\x)=\prod^k_{i=1}\Gamma^{-1}_-(x_i)\Gamma_+(x_i).
  \end{align}
 By computing $\langle0|\Gamma_B(\x)|0\rangle$ and $\langle0|\Gamma_D(\x)|0\rangle$, we can get another two dual Littlewood identities.
\end{proof}
\section{measures related to dual Littlewood identities}
Using the results in section 3, we obtain vertex operator representations of measures $\mathfrak{M}_{B}(\la)$, $\mathfrak{M}_{C}(\la)$ and $\mathfrak{M}_{D}(\la)$, and prove that they are determinantal measures.
\begin{lemma}
For generalized partitions $\la=(\la_1,\dots,\la_k)$ and $\mu=(\mu_1,\dots,\mu_k)$, one has that
\begin{align}\label{invariant-action}
(Y^*_{\mu_k-k+1}Y_{\mu_k-k+1})\cdots (Y^*_{\mu_2-1}Y_{\mu_2-1})(Y^*_{\mu_1}Y_{\mu_1})|\underline{\la}^{sp}\rangle=\delta_{\la\mu}|\underline{\la}^{sp}\rangle.
\end{align}
\end{lemma}
\begin{proof} From the commutation relations \eqref{anti-com} it follows that
\begin{align}\label{invariant operator}
(Y^*_pY_p)Y^*_q=\delta_{pq}Y^*_q+Y^*_q(Y^*_{p+1}Y_{p+1}).
\end{align}
For the case $\mu_1\geq \la_1$, the equation \eqref{invariant operator} tells us that
\begin{align}\label{invariant}
\notag(Y^*_{\mu_1}Y_{\mu_1})Y^*_{\la_1}\cdots Y^*_{\la_k}|0\rangle=&\delta_{\la_1\mu_1}Y^*_{\la_1}\cdots Y^*_{\la_k}|0\rangle+Y^*_{\la_1}\cdots Y^*_{\la_k}(Y^*_{\mu_1+k}Y_{\mu_1+k})|0\rangle\\
=&\delta_{\la_1\mu_1}Y^*_{\la_1}\cdots Y^*_{\la_k}|0\rangle
\end{align}
due to the fact $Y_{\mu_1+k}|0\rangle=0$ from \eqref{vacuum-action}.
For $\mu_1<\la_1$, using \eqref{invariant operator} repeatedly we have that
$$
(Y^*_{\mu_1}Y_{\mu_1})Y^*_{\la_1}\cdots Y^*_{\la_k}|0\rangle=\begin{cases}
0\qquad \mu_1+i-1\neq \lambda_i ~\text{for all}~ 2\leq i\leq k\\
Y^*_{\la_1}\cdots Y^*_{\la_k}|0\rangle, ~~~~ \mu_1+i-1= \lambda_i ~\text{for some $i$}
\end{cases}
$$
Repeating the process, we have that
$(Y^*_{\mu_k-k+1}Y_{\mu_k-k+1})\cdots (Y^*_{\mu_2-1}Y_{\mu_2-1})(Y^*_{\mu_1}Y_{\mu_1})|\underline{\la}^{sp}\rangle$ is nonzero for
\begin{align}
\mu_i+1\leq \la_{i+1},~~~~~~~~~~~1\leq i\leq k-1.
\end{align}
The condition $\mu_k<\la_k$ forces that
\begin{align*}
(Y^*_{\mu_k-k+1}Y_{\mu_k-k+1})Y^*_{\la_1}\cdots Y^*_{\la_k}|0\rangle=Y^*_{\la_1}\cdots Y^*_{\la_k}(Y^*_{\mu_k+1}Y_{\mu_k+1})|0\rangle=0,
\end{align*}
i.e., the left side of \eqref{invariant-action} equals 0 for the case $\mu_1<\la_1$. Combining the two cases, we complete the proof.
\end{proof}
From type-C dual Littlewood identity \eqref{eq:littlewood C}, we can consider the following measure
\begin{align}
\mathfrak{M}^{\prime}_{C}(\la)=\frac{(-1)^{\frac{|\la|}{2}}s_{\la}(\x)}{\prod_{1 \leq i\leq j \leq k}(1-x_ix_j)}
\end{align}
on partitions $\lambda$ of the type $\lambda=(\alpha+1|\alpha)$.
Let
\begin{align}
J(z)=\prod^k_{i=1}(1-x_iz)(1-x_iz^{-1}).
\end{align}
\begin{theorem} For $\la=(\la_1,\dots,\la_i,\dots,\la_k)=(n_1,\dots,n_i+i-1,\dots,n_k+k-1)=(\alpha+1|\alpha)$, one has that
\begin{align}
\mathfrak{M}^{\prime}_{C}(\la)=\det\left(K^{\prime}_C(n_i,n_j)\right)_{1\leq i,j\leq k}
\end{align}
with the kernel
\begin{align}\label{eq:C-type kernel}
K^{\prime}_C(a,b)=\int_{z}\int_{w}\frac{1-z^2}{(1-wz)(1-wz^{-1})}\frac{J(z)}{J(w)}\frac{dzdw}{(2\pi \mathbf{i})^2z^{b+1}w^{-a+1}},
\end{align}
where the integration is over the $z-$ and $w-$ counterclockwise circles around 0 satisfying $|w|<|z|<|x_i|$.
\end{theorem}
\begin{proof}
Let
\begin{align*}
\Gamma_{-}(\x)Y^*_n\Gamma^{-1}_{-}(\x)=\psi^*_n, ~~~~~\Gamma_{-}(\x)Y_n\Gamma^{-1}_{-}(\x)=\psi_n.
\end{align*}
From \eqref{anti-com} and the definitions of vertex operators $Y^*(z)$ and $Y(z)$, it is easy to check that
\begin{align}
&\psi_i\psi_j+\psi_{j+1}\psi_{i-1}=0,\qquad \psi^*_i\psi^*_j+\psi^*_{j-1}\psi^*_{i+1}=0,\qquad \psi_i\psi^*_j+\psi^*_{j+1}\psi_{i+1}=\delta_{i,j},\\
&\psi(z)=\sum_{n\in \mathbb{Z}}\psi_nz^{-n}=\Gamma_{-}(\x)Y(z)\Gamma^{-1}_{-}(\x),~~~~\psi^*(z)=\sum_{n\in \mathbb{Z}}\psi^*_nz^{n}=\Gamma_{-}(\x)Y^*(z)\Gamma^{-1}_{-}(\x).
\end{align}
Let $[z^k]f$ denote the coefficient of $z^k$ in $f$. From \eqref{eq:LittlewoodC-generating} and \eqref{invariant-action}, we have
\begin{align}\label{eq:C-measure}
\notag\mathfrak{M}^{\prime}_{C}(\la)
\notag=&\frac{1}{\prod_{1 \leq i\leq j \leq k}(1-x_ix_j)}\langle 0|(Y^*_{\la_k-k+1}Y_{\la_k-k+1})\cdots (Y^*_{\la_1}Y_{\la_1})\Gamma_C(\x)|0\rangle\\
\notag=&\langle 0|(Y^*_{\la_k-k+1}Y_{\la_k-k+1})\cdots (Y^*_{\la_1}Y_{\la_1})\Gamma^{-1}_{-}(\x)|0\rangle\\
\notag=&\langle 0|\Gamma_{-}(\x)(Y^*_{\la_k-k+1}Y_{\la_k-k+1})\cdots (Y^*_{\la_1}Y_{\la_1})\Gamma^{-1}_{-}(\x)|0\rangle\\
\notag=&\langle 0|(\psi^*_{\la_k-k+1}\psi_{\la_k-k+1})\cdots (\psi^*_{\la_1}\psi_{\la_1})|0\rangle\\
=&\left[\frac{z^{\la_1}_1z^{\la_2-1}_2\cdots z^{\la_k-k+1}_k}{w^{\la_1}_1w^{\la_2-1}_2\cdots w^{\la_k-k+1}_k}\right]\langle 0|\left(\psi^*(z_k)\psi(w_k))\cdots (\psi^*(z_1)\psi(w_1)\right)|0\rangle.
\end{align}
By the Baker-Campbell-Hausdorff formula \eqref{e:BCH}, the one point correlation is
\begin{align}
\langle 0|\psi^*(z)\psi(w)|0\rangle
=\frac{(1-z^2)}{(1-wz^{-1})(1-wz)}\frac{J(z)}{J(w)}.
\end{align}

Similarly, we have the $k$-point correlation:
\begin{align*}
&\langle 0|\left(\psi^*(z_k)\psi(w_k))\cdots (\psi^*(z_1)\psi(w_1)\right)|0\rangle\\
&~~~=\prod_{1\leq i<j\leq k}(1-z_iz_j)(1-w_iw_j)(1-z_iz^{-1}_j)(1-w_iw^{-1}_j)\\
&~~~~~\times \prod_{i\leq j}\frac{1}{(1-w_iz_j)(1-w_iz^{-1}_j)}\prod_{i< j}\frac{1}{(1-w_jz_i)(1-w^{-1}_jz_i)}\prod_{i}(1-z^2_i)\prod_{i}\frac{J(z_i)}{J(w_i)}\\
&~~~=\prod_{1\leq i<j\leq k}(1-z_iz_j)(1-w_iw_j)(z^{-1}_i-z^{-1}_j)(w_i-w_j) \prod_{i,j}\frac{1}{(1-w_iz_j)(1-w_iz^{-1}_j)}\prod_{i}(1-z^2_i)\frac{J(z_i)}{J(w_i)}\\
&~~~=\prod_{1\leq i<j\leq k}(1-\frac{1}{z_iz_j})(1-w_iw_j)(z_i-z_j)(w_i-w_j) \prod_{i,j}\frac{1}{(1-w_iz_j)(1-w_iz^{-1}_j)}\prod_{i}(1-z^2_i)\frac{J(z_i)}{J(w_i)}\\
&~~~=\det\left(\frac{(1-z^2_j)}{(1-w_iz_j)(1-w_iz^{-1}_j)}\frac{J(z_j)}{J(w_i)}\right)_{1\leq i,j\leq k},
\end{align*}
where we have used the $BC$-type Cauchy determinant \cite[(2.5)]{Be2020}:
\begin{align*}
\det\left(\frac{1}{(1-w_iz_j)(1-w_iz^{-1}_j)}\right)=\frac{\prod_{i<j}(1-\frac{1}{z_iz_j})(1-w_iw_j)(z_i-z_j)(w_i-w_j)}{\prod_{i,j}(1-w_iz_j)(1-w_iz^{-1}_j)}.
\end{align*}
From \eqref{eq:C-measure} and using contour integrals to perform the coefficient extraction, we get
\begin{align}
\mathfrak{M}^{\prime}_{C}(\la)=&\int_{z_1}\int_{w_1}\cdots \int_{z_k}\int_{w_k}\prod^k_{i=1}\frac{w_i^{n_i-1}dz_idw_i}{(2\pi \mathbf{i})^{2k}z_i^{n_i+1}}\det\left(\frac{(1-z^2_j)}{(1-w_iz_j)(1-w_iz^{-1}_j)}\frac{J(z_j)}{J(w_i)}\right)_{1\leq i,j\leq k}\\
=&\det\left(\int_{z}\int_{w}\frac{1-z^2}{(1-wz)(1-wz^{-1})}\frac{J(z)}{J(w)}\frac{dzdw}{(2\pi \mathbf{i})^2z^{n_j+1}w^{-n_i+1}}\right)_{1\leq i,j\leq k},
\end{align}
where the contours in the first $2k$-fold integral are simple counterclockwise circles centered around the origin satisfying
$|w_1|<|z_1|<\cdots<|w_k|<|z_k|<|x_i|$.
\end{proof}
\begin{remark}
Relation \eqref{eq:C-type kernel} tells us that the generating function of $K^{\prime}_C(i,j)$ is
\begin{align}
K^{\prime}_C(z,w)=\sum_{i,j\in \mathbb{Z}}K^{\prime}_C(i,j)w^{-i}z^{j}=\frac{1-z^2}{(1-wz)(1-wz^{-1})}\frac{J(z)}{J(w)}.
\end{align}
Let $t_n=\frac{1}{n}\sum^k_{i=1}x^n_i$, then $J(z)=\exp\left(-\sum_{n\geq 1}t_n(z^n+z^{-n})\right)$. Direct calculation gives that
\begin{align}
\frac{\partial}{\partial t_1}K^{\prime}_C(z,w)=\frac{1-z^2}{w}\frac{J(z)}{J(w)}.
\end{align}
\end{remark}
We can also define the following two measures
\begin{align}
&\mathfrak{M}_{B}(\la)=\frac{(-1)^{\frac{|\la|+\rk\la}{2}} s_\la(\x)}{\prod^k_{i}(1-x_i)\prod_{1 \leq i< j \leq k}(1-x_ix_j)},~~~~~\la=(\alpha|\alpha)\\
&\mathfrak{M}^{\prime}_{D}(\la)=\frac{(-1)^{\frac{|\la|}{2}}s_\la(\x)}{\prod_{1 \leq i< j \leq k}(1-x_ix_j)},~~~~~~~~~~~~~~~~~~~~\la=(\alpha|\alpha+1)
\end{align}
over respective partitions specified above.
If we replace $\x=(x_1,\dots,x_k)$ by $\mathbf{i}\x=(\mathbf{i}x_1,\dots,\mathbf{i}x_k)$, we can show that measures $\mathfrak{M}_{C}(\la)$ and $\mathfrak{M}_{D}(\la)$ are determinantal point processes (see \eqref{eq:Littlewood C measure} and \eqref{eq:Littlewood D measure} for the definitions of $\mathfrak{M}_{C}(\la)$ and $\mathfrak{M}_{D}(\la)$).
\begin{theorem}\label{thm:CD-measures}
For $\la=(\alpha+1|\alpha)$ or $\la=(\alpha|\alpha+1)$, let $n_i=\la_i-i+1$. We have
\begin{align}
&\mathfrak{M}_{C}(\la)=\det\left(K_C(n_i,n_j)\right)_{1\leq i,j\leq k},~~~~~~~~~~~~\la=(\alpha+1|\alpha),\\
&\mathfrak{M}_{D}(\la)=\det\left(K_D(n_i,n_j)\right)_{1\leq i,j\leq k},~~~~~~~~~~~~\la=(\alpha|\alpha+1),
\end{align}
where the correlation kernels are given by
\begin{align}\label{eq:type-C kernel}
&K_C(a,b)=\int_{z}\int_{w}\frac{1-z^2}{(1-wz)(1-wz^{-1})}\frac{H(z)}{H(w)}\frac{dzdw}{(2\pi \mathbf{i})^2z^{b+1}w^{-a+1}},\\
&K_D(a,b)=\int_{z}\int_{w}\frac{1-w^2}{(1-wz)(1-wz^{-1})}\frac{H(z)}{H(w)}\frac{dzdw}{(2\pi \mathbf{i})^2z^{b+1}w^{-a+1}}
\end{align}
with
\begin{align}
H(z)=\prod^k_{i=1}(1-\mathbf{i}x_iz)(1-\mathbf{i}x_iz^{-1})
\end{align}
and the $z-$ and $w-$ contours are simple counterclockwise circles around 0 satisfying $|w|<|z|<|x_i|$.
\end{theorem}

\begin{theorem}\label{thm:B-measure} Let $n_i=\la_i-i+1$ for $\la=(\alpha|\alpha)$.The measure $\mathfrak{M}_{B}(\la)$  is
a determinantal ensemble
\begin{align}
&\mathfrak{M}_{B}(\la)=\det\left(K_B(n_i,n_j)\right)_{1\leq i,j\leq k},
\end{align}
with the correlation kernel
\begin{align}
&K_B(a,b)=\int_{z}\int_{w}\frac{(1+z)(1-w)}{(1-wz)(1-wz^{-1})}\frac{J(z)}{J(w)}\frac{dzdw}{(2\pi \mathbf{i})^2z^{b+1}w^{-a+1}}.
\end{align}
\end{theorem}
\begin{remark}In fact, if the vertex operators \eqref{ver-BCD} are equipped with certain middle terms, the modes of vertex operators satisfy anti-commutation relations. From the $q$-series identity $\sum_{\lambda=(\alpha+1|\alpha)}q^{|\lambda|}=\prod^\infty_{i=1}(1+q^{2i})$, one may consider the correlation function
\begin{align}
\frac{\sum_{\lambda=(\alpha+1|\alpha)}\prod^n_{k=1}(\sum^\infty_{i=1}t^{\la_i-i+\frac{1}{2}}_k)q^{|\lambda|}}{\prod^\infty_{i=1}(1+q^{2i})}
\end{align}
similar to the case of partitions \cite{Oko2001}, strict partitions \cite{Wang2004} or self-conjugate partitions\cite{WY2026}.
\end{remark}
\begin{remark}
The third and fourth named authors together with collaborators also showed that $\mathfrak{M}_{B}(\la)$, $\mathfrak{M}_{C}(\la)$ and $\mathfrak{M}_{D}(\la)$ are determinantal in \cite[Theorem 3.12]{JLPWY2025} with different kernels.
\end{remark}
\section{one-column Littlewood-type identities} In this section, we consider three families of generalized Littlewood identities over partitions shifted by one-columns.
\begin{lemma}\label{ortho:le1} Fix $m\in\mathbb Z_+$. Let partitions $\eta=(1^m)$ and $\la=(\la_1,\dots,\la_k)$. Then
\begin{equation}
\label{column-parB}
\langle \eta^{so}|\underline{\la}^{so}\rangle=
\begin{cases}
(-1)^{\frac{|\la|+\rk\la-\delta-m}{2}+m\rk\la} & \la=(\alpha+2m,m-\delta|\alpha,0)~~~~~~~\delta\in \{0,1\}\\
0 & \text{otherwise}
\end{cases},
\end{equation}
\begin{equation}
\label{column-parC}
\langle \eta^{sp}|\underline{\la}^{sp}\rangle=
\begin{cases}
(-1)^{\frac{|\la|-m}{2}+m\rk\la} & \la=(\alpha+2m+1,m+1|\alpha,0)~~\text{or}~~(\alpha+2m+1,m-1|\alpha,0)\\
0 & \text{otherwise}
\end{cases},
\end{equation}
\begin{equation}
\label{column-parD}
\langle \eta^{o}|\underline{\la}^{o}\rangle=
\begin{cases}
(-1)^{\frac{|\la|-m}{2}+m\rk\la} & \la=(\alpha+2m-1,m-1|\alpha,0)\\
0 & \text{otherwise}
\end{cases},
\end{equation}
where we have adopted the shorthand notation:
$$(\alpha+2m+1,m+1|\alpha,0)=(\alpha_1+2m+1,\dots,\alpha_{r-1}+2m+1,m+1|\alpha_1,\dots,\alpha_{r-1},0).$$
\end{lemma}
\begin{proof}From the proof of Lemma \ref{le:classical Littlewood}, the inner product
\begin{align}
\langle \eta^{sp}|\underline{\la}^{sp}\rangle=\langle 0|(Y^{*}_{-1})^mY^{*}_{\la_1}Y^{*}_{\la_2}\cdots Y^{*}_{\la_k}|0\rangle
\end{align}
is nonzero only for
\begin{align}\label{nonzero-column-C}
\langle 0|(Y^{*}_{-1})^mY^{*}_{\la_1}Y^{*}_{\la_2}\cdots Y^{*}_{\la_k}=\epsilon\langle 0|(Y^{*}_0)^{k+m},
\end{align}
where $\epsilon\in \{1,-1\}$. From \eqref{permutation-C} and \eqref{nonzero-column-C}, we have
\begin{align}
&\la_j=\sigma(k+1-j)-(k+1-j) ~~\text{or}~~-\sigma(k+1-j)+k+2m+1+j,~~~~~~~~1\leq j\leq k,\\
\label{eq:permutation-relation}&\sigma(k+i)-(k+i)=-1,~~~~~~~~~~1\leq i\leq m,
\end{align}
i.e.,
\begin{align}
\epsilon_j(\la_j-j-m)=k+m+1-\sigma(k+1-j),
\end{align}
where $\varepsilon_j\in \{1,-1\}$. From \eqref{eq:permutation-relation}, we know that $\sigma(k+1-j)\in \{1,2,\dots,k-1,k+m\}$, and thus
\begin{align}
\{\epsilon_1(\la_1-1-m),\epsilon_2(\la_2-2-m),\cdots,\epsilon_k(\la_k-k-m)\}\longleftrightarrow \{1,m+2,\cdots,m+k\}.
\end{align}
Along the lines of the proof of Lemma \ref{le:classical Littlewood}, we have $\la_1=2m+k+1$. Assuming
\begin{align}\label{eq:signC}
\epsilon_r(\la_r-r-m)=1, ~~~~~~~~2\leq r\leq k,
\end{align}
then
\begin{align}
\notag&\{\epsilon_2(\la_2-2-m),\cdots,\epsilon_{j-1}(\la_{r-1}-{r-1}-m),\epsilon_{r+1}(\la_{r+1}-{r+1}-m),\cdots,\epsilon_k(\la_k-k-m)\}\\
&~~~~~~~~~~~~\longleftrightarrow \{m+2,\cdots,m+k-1\}.
\end{align}
Using similar discussion in Lemma \ref{le:classical Littlewood}, we have
\begin{align}\label{eq:partition-relation1}
&\la^{\prime}_i+2m+1=\la_i,~~~~~~1\leq i\leq r-1,\\
\label{eq:partition-relation2}&\la_r=r+m-1~~\text{or}~~r+m+1.
\end{align}
It is easy to show that $\epsilon_{r+1}=-1$ from $\la_{r+1}\leq r+m+1$ and $\epsilon_{r+1}(\la_{r+1}-{r+1}-m)\geq m+2$. We therefore have
\begin{align}
\la_{r+1}\leq r-1,
\end{align}
in other words,
\begin{align}\label{eq:partition-relation3}
\la^{\prime}_r=r.
\end{align}
Combining \eqref{eq:partition-relation1}, \eqref{eq:partition-relation2} and \eqref{eq:partition-relation3}, we have $\langle \eta^{sp}|\underline{\la}^{sp}\rangle$ is nonzero only for
\begin{align}\label{eq:C-type paritition}
\la\in \{(\alpha+2m+1,m\pm 1|\alpha,0)| \alpha_1>\alpha_2>\cdots >\alpha_{r-1}> 0\}.
\end{align}

If partition $\la$ of the type described in \eqref{eq:C-type paritition} with $\rk\la=r$, then $\la$ can be written as
\begin{align}
\la=(\alpha_1+2m+2,\dots,\alpha_{r-1}+2m+r,r+m\pm 1,\la_{r+1},\dots,\la_{k})
\end{align}
with $\la_{r+1}\leq r-1$.
Using the commutation relation \eqref{anti-com}, we can express $(Y^{*}_{-1})^mY^{*}_{\la_1}Y^{*}_{\la_2}\cdots Y^{*}_{\la_k}$ as
\begin{align}
(-1)^{mr}Y^{*}_{\alpha_1+m+2}\cdots Y^{*}_{\alpha_{r-1}+m+r}Y^{*}_{r\pm1}(Y^{*}_{r-1})^mY^{*}_{\la_{r+1}}\cdots Y^{*}_{\la_{k}}, &~~ \la=(\alpha+2m+1,m\pm1|\alpha,0)
\end{align}
It is easy to check that $\nu=(\alpha_1+m+2,\dots,\alpha_{r-1}+m+r,r\pm 1,\underbrace{r-1,\dots,r-1}_m,\la_{r+1},\dots,\la_{k})$ is a partition and $|\nu|+m=|\la|$. Comparing with \eqref{vacuum-parC}, we can prove \eqref{column-parC}. Relations \eqref{column-parB} and \eqref{column-parD} could be proved similarly.
\end{proof}

\begin{theorem}\label{thm:Littlewood-column}For $\x=(x_1,\dots,x_k)$, one has
\begin{align}
\label{eq:littlewood-column B}
\notag&\sum_{\substack{\la=(\alpha+2m,m-\delta|\alpha,0)\\ \delta\in\{0,1\}}}(-1)^{\frac{|\la|+\rk\la+m}{2}+m\rk\la} s_\la(\x)\\
&~~~~~~=\prod^k_{i}(1-x_i)\prod_{1 \leq i< j \leq k}(1-x_ix_j)\det\left(f_{s-t+1}(\x)+f_{s+t}(\x)\right)_{1\leq s,t\leq m}  ,\\
\label{eq:littlewood-column C}
\notag&\sum_{\substack{\la=(\alpha+2m+1,m+1|\alpha,0)\\\& (\alpha+2m+1,m-1|\alpha,0)}}(-1)^{\frac{|\la|+m}{2}+m\rk\la} s_\la(\x)\\
&~~~~~~=\prod_{1 \leq i\leq j \leq k}(1-x_ix_j)\det\left(f_{s-t+1}(\x)-f_{s+t+1}(\x)\right)_{1\leq s,t\leq m} ,\\
\label{eq:littlewood-column D}
\notag&\sum_{\la=(\alpha+2m-1,m-1|\alpha,0)}(-1)^{\frac{|\la|+m}{2}+m\rk\la} s_\la(\x)\\
&~~~~~~=\frac{1}{2}\prod_{1 \leq i< j \leq k}(1-x_ix_j)\det_{1\leq s,t\leq m}\left(f_{s-t+1}(\x)+f_{s+t-1}(\x)\right),
\end{align}
where
\begin{align}
f_i(\x)=f_{-i}(\x)=\sum^{k-i}_{j=0}e_j(\x)e_{i+j}(\x).
\end{align}
\end{theorem}
\begin{proof} We first prove \eqref{eq:littlewood-column C} in detail.
From the definitions of vertex operators $Y^{*}(z)$ \eqref{ver-BCD} and $\Gamma^{-1}_-(\x)$ \eqref{eq:half-vertex}, we have
\begin{align}
Y^{*}(z)\Gamma^{-1}_-(\x)=\Gamma^{-1}_-(\x)Y^{*}(z)\prod^k_{i=1}(1-x_iz^{-1})(1-x_iz),
\end{align}
In terms of components
\begin{align}\label{half-vertexC}
Y^{*}_n\Gamma^{-1}_-(\x)=\sum_{i\in\mathbb{Z}}(-1)^if_i(\x)\Gamma^{-1}_-(\x)Y^{*}_{n-i}.
\end{align}
Similarly,
\begin{align}
&U^{*}_n\Gamma^{-1}_-(\x)=\sum_{i\in\mathbb{Z}}(-1)^if_i(\x)\Gamma^{-1}_-(\x)U^{*}_{n-i},\\
&W^{*}_n\Gamma^{-1}_-(\x)=\sum_{i\in\mathbb{Z}}(-1)^if_i(\x)\Gamma^{-1}_-(\x)W^{*}_{n-i}.
\end{align}
Using the Baker-Campbell-Hausdorff formula \eqref{e:BCH} and \eqref{half-vertexC}, we have
\begin{align}\label{columnC-right}
\notag&\langle 0|(Y^{*}_{-1})^m\Gamma_C(\x)|0\rangle=\prod_{1\leq i\leq j\leq k}(1-x_ix_j)\langle 0|(Y^{*}_{-1})^m\Gamma^{-1}_-(\x)|0\rangle\\
\notag&~~~~~=\prod_{1\leq i\leq j\leq k}(1-x_ix_j)\sum_{i_1,\dots,i_m\in\mathbb{Z}}(-1)^{i_1+\cdots+i_m}f_{i_m}(\x)\cdots f_{i_1}(\x)\langle 0|Y^{*}_{-1-i_m}\cdots Y^{*}_{-1-i_1}|0\rangle\\
&~~~~~~=\prod_{1\leq i\leq j\leq k}(1-x_ix_j)\det\left((-1)^{i-j-1}(f_{i-j-1}(\x)-f_{i+j-2m-3}(\x))\right)_{1\leq i,j\leq m}\\
\notag&~~~~~~=\prod_{1\leq i\leq j\leq k}(1-x_ix_j)(-1)^m\det\left(f_{i-j-1}(\x)-f_{i+j-2m-3}(\x)\right)_{1\leq i,j\leq m}\\
\notag&~~~~~~=\prod_{1\leq i\leq j\leq k}(1-x_ix_j)(-1)^m\det\left(f_{s-t+1}(\x)-f_{s+t+1}(\x)\right)_{1\leq t,s\leq m},
\end{align}
where we have used \eqref{permutation-C} and the properties of determinants.
From \eqref{eq:LittlewoodC-generating} and \eqref{column-parC}, we have
\begin{align}\label{columnC-left}
\notag&\langle 0|(Y^{*}_{-1})^m\Gamma_C(\x)|0\rangle\\
\notag&~~~~~=\sum_{\substack{\la\\l(\la)\leq k}}s_\la(\x)\langle 0|(Y^{*}_{-1})^m|\underline{\la}^{sp}\rangle\\
&~~~~~=\sum_{\substack{\la=(\alpha+2m+1,m+1|\alpha,0)\\\& (\alpha+2m+1,m-1|\alpha,0)}}(-1)^{\frac{|\la|-m}{2}+m\rk\la} s_\la(\x).
\end{align}
Comparing \eqref{columnC-left} and \eqref{columnC-right}, we get \eqref{eq:littlewood-column C}.

From the vertex operators $\Gamma_B(\x)$ \eqref{eq:half-vertexB} and $\Gamma_D(\x)$ \eqref{eq:half-vertexD}, we can similarly obtain \eqref{eq:littlewood-column B} and \eqref{eq:littlewood-column D} by evaluating $\langle 0|(U^{*}_{-1})^m\Gamma_B(\x)|0\rangle$ and $\langle 0|(W^{*}_{-1})^m\Gamma_D(\x)|0\rangle$ with the help of the inner-product formulas \eqref{column-parB} and \eqref{column-parD}.
\end{proof}
\begin{remark}
The proofs in this section show that we can obtain three generalized Littlewood identities for any fixed generalized partition.
\end{remark}

{
\section{generalized identities for $(-n)$-asymmetric partitions} By computing the Euler characteristic of the BGG resolutions, Erickson and Hunziker obtained some new families of generalized Littlewood identities for $(-n)$-asymmetric partitions\cite{EH2026}. For any positive integer $n$, we also provide generalized Littlewood identities for $(-n)$-asymmetric partitions from the vertex algebraic viewpoint.

\begin{lemma} Fix $m\in\mathbb Z_+$. For the generalized partitions $\eta=(0^m)$ and $\la=(\la_1,\dots,\la_k)$, one has that
\begin{equation}
\label{empty-parB}
\langle \eta^{so}|\underline{\la}^{so}\rangle=
\begin{cases}
(-1)^{\frac{|\la|+\rk\la}{2}+m\rk\la} & \la=(\alpha+2m|\alpha)\\
0 & \text{otherwise}
\end{cases},
\end{equation}
\begin{equation}
\label{empty-parC}
\langle \eta^{sp}|\underline{\la}^{sp}\rangle=
\begin{cases}
(-1)^{\frac{|\la|}{2}+m\rk\la} & \la=(\alpha+2m+1|\alpha)\\
0 & \text{otherwise}
\end{cases},
\end{equation}
\begin{equation}
\label{empty-parD}
\langle \eta^{o}|\underline{\la}^{o}\rangle=
\begin{cases}
(-1)^{\frac{|\la|}{2}+m\rk\la} & \la=(\alpha+2m-1|\alpha)\\
0 & \text{otherwise}
\end{cases}.
\end{equation}
\end{lemma}
\begin{proof}
Similar to the proof of Lemma \ref{ortho:le1}.
\end{proof}
\begin{theorem}\label{th:asy1}For $\x=(x_1,\dots,x_k)$, one has
\begin{align}
\label{eq:littlewood-empty B}
\notag&\sum_{\la=(\alpha+2m|\alpha)}(-1)^{\frac{|\la|+\rk\la}{2}+m\rk\la} s_\la(\x)\\
&~~~~~~=\prod^k_{i}(1-x_i)\prod_{1 \leq i< j \leq k}(1-x_ix_j)\det\left(f_{s-t}(\x)+f_{s+t-1}(\x)\right)_{1\leq s,t\leq m}  ,\\
\label{eq:littlewood-empty C}
\notag&\sum_{\la=(\alpha+2m+1|\alpha)}(-1)^{\frac{|\la|}{2}+m\rk\la} s_\la(\x)\\
&~~~~~~=\prod_{1 \leq i\leq j \leq k}(1-x_ix_j)\det\left(f_{s-t}(\x)-f_{s+t}(\x)\right)_{1\leq s,t\leq m} ,\\
\label{eq:littlewood-empty D}
\notag&\sum_{\la=(\alpha+2m-1|\alpha)}(-1)^{\frac{|\la|}{2}+m\rk\la} s_\la(\x)\\
&~~~~~~=\frac{1}{2}\prod_{1 \leq i< j \leq k}(1-x_ix_j)\det\left(f_{s-t}(\x)+f_{s+t-2}(\x)\right)_{1\leq s,t\leq m}.
\end{align}
\end{theorem}
\begin{proof}
By computing $\langle 0|(U^{*}_{0})^m\Gamma_B(\x)|0\rangle$, $\langle 0|(Y^{*}_{0})^m\Gamma_C(\x)|0\rangle$ and $\langle 0|(W^{*}_{0})^m\Gamma_D(\x)|0\rangle$, we can complete the proof by using similar process to prove Theorem \ref{thm:Littlewood-column}.
\end{proof}
\begin{remark} Using the involution $\omega$ that $\omega(s_{\la}(\x))=s_{\la'}(\x)$ and $\omega\left(\prod_{1 \leq i\leq j \leq k}(1-x_ix_j)\right)=\prod_{1 \leq i< j \leq k}(1-x_ix_j)$, one can get generalized Littlewood identities for $n$-asymmetric partitions.
\end{remark}
\begin{remark} The measures provided by Theorem \ref{th:asy1} may lead to more general results than Theorems \ref{thm:CD-measures} and \ref{thm:B-measure},
which will be studied in our forthcoming paper.
\end{remark}

\section{Connections with some known bounded Littlewood identities}
In \cite{HKKO2025b}, Huh, Kim, Krattenthaler and Okada obtained some bounded Littlewood identities. In the following we point out
some connections with our generalized Littlewood identities (cf. and \cite{EH2026}).

\subsection*{Case 1} In \cite{HKKO2025b}, a special case of (5.3) could be stated as follows.
\begin{align}
\sum_{\substack{\la_1\leq 2m\\ c(\la)=m}}s_\lambda(\x)=\frac{1}{2}\det\left(f_{s-t+1}(\x)+f_{s+t-1}(\x)\right)_{1\leq s,t\leq m},
\end{align}
where $c(\la)$ denotes the number of odd columns of $\lambda$. Comparing with \eqref{eq:littlewood-column D}, we have
\begin{align}\label{e:conn1}
\sum_{\la=(\alpha+2m-1,m-1|\alpha,0)}(-1)^{\frac{|\la|+m}{2}+m\rk\la} s_\la(\x)
=\prod_{1 \leq i< j \leq k}(1-x_ix_j)\sum_{\substack{\la_1\leq 2m\\ c(\la)=m}}s_\lambda(\x).
\end{align}
In fact, \eqref{e:conn1} is the generalized Littlewood identity (6.10) for $a=b=m$ in \cite{EH2026}.

\subsection*{Case 2} Theorem 1.1 in \cite{HKKO2025b} can be rewritten as follows (cf. \cite[Theorem 7.1]{Ste1990}, \cite[Theorem 1.1]{HKKO2025a})
\begin{align}
&\sum_{\lambda_1\leq 2m+1}s_{\la}(\x)=\prod^k_{i=1}(1+x_i)\det\left(f_{s-t}(\x)-f_{s+t}(\x)\right)_{1\leq s,t\leq m},\\
&\sum_{\lambda_1\leq 2m}s_{\la}(\x)=\det\left(f_{s-t}(\x)+f_{s+t-1}(\x)\right)_{1\leq s,t\leq m}.
\end{align}
Comparing with Theorem \ref{th:asy1}, we have
\begin{align}\label{e:pro1}
&\sum_{\la=(\alpha+2m+1|\alpha)}(-1)^{\frac{|\la|}{2}+m\rk\la} s_\la(\x)=\prod^k_{i}(1-x_i)\prod_{1 \leq i< j \leq k}(1-x_ix_j)\sum_{\lambda_1\leq 2m+1}s_{\la}(\x),\\
\label{e:pro2}&\sum_{\la=(\alpha+2m|\alpha)}(-1)^{\frac{|\la|+\rk\la}{2}+m\rk\la} s_\la(\x)=\prod^k_{i}(1-x_i)\prod_{1 \leq i< j \leq k}(1-x_ix_j)\sum_{\lambda_1\leq 2m}s_{\la}(\x).
\end{align}
\begin{remark} The
Littlewood-type identity \eqref{e:pro1} agrees with \cite[Table4(II)]{EH2026}.
\end{remark}
\begin{remark} By the classical Littlewood identity
\begin{align}\label{e:Littlewood-identity}
\sum_{\la}s_{\la}(\x)=\frac{1}{\prod^k_{i}(1-x_i)\prod_{1 \leq i< j \leq k}(1-x_ix_j)},
\end{align}
the right sides of \eqref{e:pro1} and \eqref{e:pro2} could be seen as some probability measures, and are in fact
the correlation functions $\langle 0|(Y^{*}_{0})^m\Gamma_C(\x)|0\rangle$ and $\langle 0|(W^{*}_{0})^m\Gamma_D(\x)|0\rangle$, respectively.
\end{remark}
\begin{remark} Let $\la=(\alpha+n|\alpha)$ and $|\alpha|=\alpha_1+\cdots+\alpha_{\rk \la}$. One has that $(-1)^{\frac{|\la|}{2}+m\rk\la}=(-1)^{|\alpha|+\rk \la}$ for $n=2m+1$ and $(-1)^{\frac{|\la|+\rk\la}{2}+m\rk\la}=(-1)^{|\alpha|+\rk \la}$ for $n=2m$. By the involution $\omega$, \eqref{e:pro1} and \eqref{e:pro2} become
\begin{align}\label{e:Littlewood-Lie}
\sum_{l(\lambda)\leq n}s_{\la}(\x)=\frac{\sum_{\la=(\alpha|\alpha+n)}(-1)^{|\alpha|+\rk\la} s_\la(\x)}{\prod^k_{i}(1-x_i)\prod_{1 \leq i< j \leq k}(1-x_ix_j)},
\end{align}
which will reduce to the Littlewood identity \eqref{e:Littlewood-identity} by setting $n\geq k$, and this
identity 
was conjectured by Lievens, Stoilova, Van der Jeugt\cite{LSV2008} and proved by King\cite{Kin2013}. Thus our method offers a new proof of this Littlewood identity.
\end{remark}

\subsection*{Case 3}
By \eqref{permutation-C}, we have
\begin{align*}
&\langle 0|(Y^{*}_{0})^p(Y^{*}_{-1})^q\Gamma_C(\x)|0\rangle \\
&~~~~=(-1)^q\prod_{1 \leq i\leq j \leq k}(1-x_ix_j)\det_{1 \le i, j \le w-1} \left(
  \begin{cases}
   f_{i-j}(\x) - f_{i+j}(\x), &\mbox{if \( 1 \le i \le p \)} \\
   f_{i-j+1}(\x) - f_{i+j+1}(\x), &\mbox{if \( p+1 \le i \le p+q \)}
  \end{cases}
 \right),
\end{align*}
where the determinant in the right side appeared in \cite[Theorem 1.2, Theorem1.3]{HKKO2025b}. This suggests there
would be an interesting Lie theoretical explanation from the vertex algebraic viewpoint.
We will study $\langle 0|(Y^{*}_{0})^p(Y^{*}_{-1})^q\Gamma_C(\x)|0\rangle$, $\langle 0|(U^{*}_{0})^p(U^{*}_{-1})^q\Gamma_B(\x)|0\rangle$ and $\langle 0|(W^{*}_{0})^p(W^{*}_{-1})^q\Gamma_D(\x)|0\rangle$ and explore their connections in a forthcoming paper.

}

\vskip 20pt

\noindent{\bf Acknowledgments.} We thank the anonymous referees for helpful suggestions, which have enhanced the quality of the paper. This work is supported by NSFC (grant nos.12301033, 12171303), NSF of Huzhou (grant no. 2022YZ47), the Simons Foundation (grant no. MP-TSM-00002518) and the Open Research Fund of Hubei Key Laboratory of Mathematical Sciences (Central China Normal University), Wuhan 430079, P. R. China (grant MPL2025ORG007).
\smallskip



\begin{thebibliography}{10}
\bibitem{Alb2025} S.~Albion, {\it Character factorisations, $z$-asymmetric partitions and plethysm}, arXiv:2501.18520, 2025.
\bibitem{AK2022} A.~Ayyer, N, Kumari, {\it Factorization of classical characters twisted by roots of unity}, J. Algebra 609 (2022), 437–483.
\bibitem{Ba1996} T.~H.~Baker, {\it Vertex operator realization of symplectic and orthogonal S-functions}, J. Phys. A. 29(12) (1996), 3099--3117.
\bibitem{BBC2020} G.~Barraquand, A.~Borodin, I.~Corwin, {\it Half-space Macdonald processes}, Forum Math. Pi 8 (2020), e11.
\bibitem{BBCW2018} G.~Barraquand, A.~Borodin, I.~Corwin, M.~Wheeler, {\it Stochastic six-vertex model in a half-quadrant and half-line open asymmetric simple exclusion process}, Duke Math. J. 167(13) (2018), 2457–2529.
\bibitem{Be2020} D.~Betea, {\it Determinantal point processes from symplectic and orthogonal characters and applications},
S\'emin. Lothar. Comb. 84B, Article 41, 12 p. (2020).
\bibitem{BW2016} D.~Betea, M.~Wheeler, {\it Refined Cauchy and Littlewood identities, plane partitions, and symmetry classes of alternating
sign matrices}, J. Comb. Theory, Ser. A 137 (2016), 126--165.
\bibitem{BNNS2024} D.~Betea, A.~Nazarov, P.~Nikitin, T.~Scrimshaw, {\it Limit shapes and fluctuations for ($\mathrm{GL}_n,\mathrm{GL}_k$) skew Howe duality}, Ann. Henri Poincaré (2025), to appear.
\bibitem{BZ2019} E.~Bisi, N.~Zygouras, {\it Point-to-line polymers and orthogonal Whittaker functions}, Trans. Amer.
Math. Soc. 371(12) (2019), 8339--8379.
\bibitem{BR2005} A.~Borodin, E.~M.~Rains, {\it Eynard–Mehta theorem, Schur process, and their Pfaffian analogs}, J. Stat.
Phys. 121(3-4) (2005), 291--317.
\bibitem{CM2024} C.~Cuenca, M.~Mucciconi, {\it The symplectic Schur process}, arXiv: 2407.02415, 2024.
\bibitem{DJKM1983} E.~Date, M.~Jimbo, M.~Kashiwara, T.~Miwa, {\it Transformation groups for soliton equations}.
In: Nonlinear Integrable Systems--Classical Theory
and Quantum Theory, pp. 39-119. World Sci. Publishing, Singapore, 1983.
\bibitem{EH2026} W.~Q.~Erickson, M. Hunziker, {\it Dimension identities, almost self-conjugate partitions, and BGG complexes for hermitian symmetric pairs}, J. Comb. Theory, Ser. A 219 (2026), 106118.
\bibitem{FK1980}I.~B.~Frenkel, V.~G.~Kac, {\it Basic representations of affine Lie algebras and dual resonance models}, Invent. Math. 62 (1980), 23--66.
\bibitem{FH1991} W.~Fulton, J.~Harris, {\it Representation Theory: A First Course}, Graduate Texts in Mathematics,
vol. 129, Springer-Verlag, New York, 1991.
\bibitem{GKS1990} F.~Garvan, D.~Kim, D.~Stanton, {\it Cranks and t-cores}, Invent. Math. 101(1) (1990), 1--17.
\bibitem{HKKO2025a} J.~Huh, J. S.~Kim, C.~Krattenthaler, S.~Okada, {\it Bounded Littlewood identities for cylindric Schur functions}, Trans. Amer.
Math. Soc. 378 (2025), 6765–6829.
\bibitem{HKKO2025b} J.~Huh, J.~S.~Kim, C.~Krattenthaler, S.~Okada, {\it Bounded Littlewood identities with
fixed number of odd rows or odd columns}, arXiv: 2510.11054.
\bibitem{KRR2013}V.~G. Kac, A.~K. Raina, N. Rozhkovskaya, {\it Bombay Lectures on Highest Weight Representations of
Infinite Dimensional Lie Algebras}, 2nd edn. World Scientific, Hackensack, 2013.
\bibitem{Jing1991} N.~Jing, {\it Vertex operators, symmetric functions, and the spin group $\Gamma_{n}$}, J. Algebra 138(2) (1991), 340--398.
\bibitem{JL2022} N.~Jing, Z,~Li, {\it A note on Cauchy's formula}, Adv. Appl. Math. 153 (2024), Article ID 102630. 
\bibitem{JLW2024}N.~Jing, Z.~Li, D.~Wang, {\it Skew symplectic and orthogonal Schur functions}, SIGMA 20 (2024), 041.
\bibitem{JLPWY2025} N.~Jing, Z.~Li, X.~Pan, D.~Wang, C.~Ye, {\it Skew odd orthogonal characters and interpolating Schur polynomials}, Bull. London Math. Soc. 57(8) (2025), 2509--2530.
\bibitem{JN2015} N.~Jing, B,~Nie, {\it Vertex operators, Weyl determinant formulae and Littlewood duality}, Ann. Combin. 19(3) (2015), 427--442.
\bibitem{JR2016} N.~Jing, N.~Rozhkovskaya, {\it Vertex operators arising from Jacobi-Trudi identities}, Comm. Math. Phys. 346 (2016), 679--701.
\bibitem{JW2025} F.~Jouhet, D.~Wahiche, {\it Congruences for hook lengths of partitions}, arXiv:2502.06423, 2025.
\bibitem{Kin2013} R.~C. King, {\it From Palev's study of Wigner quantum systems to new results on sums of Schur functions}, In: Dobrev, V. (eds) Lie Theory and Its Applications in Physics. Springer Proceedings in Mathematics $\&$ Statistics, vol 36, Springer, Tokyo, 2013.
\bibitem{LSV2008} S.~Lievens, N.I.~Stoilova, J.~Van der Jeugt, {\it The paraboson Fock space and unitary irreducible
representations of the Lie superalgebra $\mathfrak{osp}(1|2n)$}, Commun. Math. Phys. 281 (2008), 805--826.
\bibitem{Lit1950}D. E. Littlewood, {\it The theory of group characters and matrix representations of groups}, 2nd ed. Oxford University Press, London, 1950.
\bibitem{Mac1995} I.~G.~Macdonald, {\it Symmetric functions and Hall polynomials}, 2nd edition, Oxford University Press, Oxford, 1995.
\bibitem{Me2019} E.~Meckes, {\it The random matrix theory of the classical compact groups}, Cambridge University Press, 2019. 
\bibitem{Oko2001} A. Okounkov, {\it Infinite wedge and random partitions}, Selecta Math. 7 (2001), 57--81.
\bibitem{Rai2000}E.~M.~Rains, {\it Correlation functions for symmetrized increasing subsequences}, arXiv: math/0006097 [math.CO], 2000.
\bibitem{Sta1999} R.~P. Stanley, {\it Enumerative combinatorics}, vol. 2, Cambridge University Press, Cambridge, 1999.
\bibitem{Ste1990}J.~ R.~Stembridge, {\it Nonintersecting paths, Pfaffians, and plane partitions}, Adv. Math. 83(1) (1990), 96--131.
\bibitem{Wang2004} W.~Wang, {\it Correlation functions of strict partitions and twisted Fock spaces}, Transform. Groups 9(1) (2004), 89--101.
\bibitem{WY2026} Z.~Wang, C.~Yang, {\it Correlation function of self-conjugate partitions: $q$-difference equation and quasimodularity}, Math. Ann. 395, (2026), 50.
\bibitem{Wey1946} H.~Weyl, {\it Classical groups: their invariants and representations}, Princeton University
Press, Princeton, 1946.
\end{thebibliography}
\end{document}